\documentclass[twoside,11pt,a4paper]{article}
\usepackage{pifont}
\usepackage{bbm}
\usepackage{bbding}
\usepackage{latexsym}
\usepackage{amsmath}   
\usepackage{tipa}
\usepackage{mathrsfs}  
\usepackage{amssymb}   
\usepackage{color}     
\usepackage{caption2}  
\usepackage{amsfonts}
\usepackage{indentfirst}
\usepackage{amsbsy}
\usepackage{amsthm}
\usepackage{bm}
\usepackage{hyperref}
\usepackage{geometry}
\def\beeq{\begin{equation}}      \def\eneq{\end{equation}}
\def\beeqy{\begin{eqnarray}}     \def\eneqy{\end{eqnarray}}
\def\bece{\begin{center}}        \def\ence{\end{center}}
\def\ba {\begin{array}}          \def\ea {\end{array}}
\def\bess{\begin{eqnarray*}}     \def\eess{\end{eqnarray*}}
\def\bes{\begin{split}}          \def\ens{\end{split}}
\def\bali{\begin{align}}         \def\enali{\end{align}}
\def\bali{\begin{remark}}         \def\enali{\end{remark}}

\def\s1{\sqrt{-1}}

\theoremstyle{plain}

\newtheorem{theorem}{\noindent{\bf Theorem}}[section]

\newtheorem{lem}[]{\noindent{\bf Lemma}}[section]

\newtheorem{rem}[]{\noindent{\bf Remark}}[section]

        \def\bin#1#2 {{#1\choose#2}}

           \def\dfrac#1#2 {{\displaystyle{#1\over#2}}}
           \def\din#1#2 {{\displaystyle{#1\choose#2}}}

\def\a.{{\rm \ddot{a}}}

\begin{document}
\renewcommand{\thesection}{\arabic{section}}
\renewcommand{\theequation}{\thesection.\arabic{equation}}

\baselineskip=18pt

\title {{\bf Constructing constant curvature  metrics on Riemann surfaces with  singularities }
\author {\small\textit{ Zhiqiang Wei $^*$}\\
\scriptsize\textit{School of Mathematics and Statistics, Henan University, Kaifeng 475004, P.R. China}\\
\scriptsize\textit{Email: weizhiqiang15@mails.ucas.edu.cn}\\
}
\footnotetext{\scriptsize $^*$This work is supported by the National Science Foundation of Henan (Grant No.202300410047)}
}
\date{}
\maketitle

\noindent{\bf Abstract}\hskip3mm  By constructing an ODE through a kind of meromorphic 1-forms, we will give an explicit construction of a kind of conformal metrics of constant curvature on Riemann surfaces with  singularities. As an application, we will classify constant curvature one metrics on $S^{2}$ with two conical singularities, which was first proved by Troyanov in \cite{Tr89} by using projective connection.
\vspace*{2mm}


\vspace*{2mm}
\noindent{\bf 2020 MR Subject Classification:}\hskip3mm 30F45.

\thispagestyle{empty}

\section{Introduction}
\setcounter{equation}{0}
As is well known, the classical uniformation theorem says that there exists a constant scalar curvature (CSC) metric  in each fixed K$\ddot{a}$hler class of Riemann surface without a boundary. A natural question is how to generalize the classical uniformation theorem  to surfaces with boundaries. \par
A (conformal) metric $g$ on a Riemann surface $M$ has a conical singularity of singular angle $2\pi\alpha(0<\alpha\neq1)$ at a point $p\in M$ if in a small neighborhood of $p$, there exists a local complex coordinate chart $(U,z)$ with $z(p)=0$, such that $g=e^{2\varphi(z,\overline{z})}|dz|^{2}$ and
$$\varphi(z,\overline{z})-(\alpha-1)\ln|z|$$
is continuous at $0$ (see \cite{Tr89}). If $g$  is a Riemann metric on $M$ with conical singularities $p_{1},\ldots,p_{N}$ and  conical angles $2\pi\alpha_{1},\ldots,2\pi\alpha_{N}$ respectively, we say that $g$ represents the divisor $D:=\sum_{n=1}^{N}(\alpha_{n}-1)P_{n}$. And we denote the K-surface by $M_{\{\alpha_{1},\ldots,\alpha_{N}\}}$ \par
A classical problem in the theory of surfaces is,
given a divisor $D=\sum_{n=1}^{N}(\alpha_{n}-1)P_{n}~(1\neq \alpha_{n}>0,\forall n)$ on a Riemann surface $M$,
   whether there exists a conformal CSC-$K$ (Constant Scalar Curvature $K$) metric  representing $D$. If $M$ is compact, the Gauss-Bonnet formula says that the integral of the curvature on $M$ equals $2\pi$ times
$$\chi(M)+deg~D,$$
where $\chi(M)$ denotes the Euler number of $M$ and $deg~D=\sum_{n=1}^{N}(\alpha_{n}-1)$ the degree of the divisor $D$. If $K\leq0$, then the unique metric exists if and only if $\chi(M)+deg~D\leq 0$(see \cite{Mc88},\cite{Tr91}). If $\chi(M)+deg~D> 0$,  the problem is still open now, except that there are some partial results. Troyanov \cite{Tr89} proved that there is a CSC-1 metric on $S^{2}_{\{\alpha,\beta\}}$ if and only if $\alpha=\beta$. Troyanov \cite{Tr91} found a sufficient condition under which there exists a CSC-1 metric on a compact Riemann surface $M$ with finite conical singularities.  Provided that $M=S^{2}$ and all angles lie in $(0,2\pi)$, Luo and Tian \cite{LT92} proved the sufficient condition of Troyanov is also necessary and the metric is unique. Under some restrictive conditions, Chen and Li \cite{CL91} found some necessary conditions of the existence of CSC-1 metrics  on compact Riemann surfaces with finite conical singularities. The authors \cite{Xu15} studied the developing maps of a CSC-1 metric $g$ on a compact Riemann surface $M$ with finite conical singularities. In general, a developing map of $g$  is a multivalued meromorphic function on $M\setminus\{singularities\}$. In particularly, they  constructed CSC-1 metrics with finite conical singularities by using a kind of meromorphic 1-form on a compact Riemann surface. For more results about CSC metrics with  singularities, we refer the readers to \cite{DFA11,Br88,Er04,Er20,LX20,LSX21,MD16,UY00} and references cited in.\par
By the way, another way to generalize the classical uniformization theorem on surfaces with edges is by extremal K$\ddot{a}$hler metrics with singularities. We often call a non-CSC extremal metric with finite singularities on a compact Riemann surface a non-CSC HCMU(the Hessian of the Curvature of the Metric is Umbilical) metric.
 We refer the readers  to \cite{Ch00,WZ00,LZ02,Wu11,Wu15,Wei21} and references cited in for more results about non-CSC HCMU metrics. In this paper, we mainly focus on CSC metrics on Riemann surfaces with singularities.\par
 Our purpose in this paper is to construct conformal CSC-$K$ metric on Riemann surfaces with singularities. Our motivation comes form the papers \cite{Wu11} and \cite{Xu15}. By constructing an ODE through a kind of meromorphic 1-forms, we will prove the following theorem which can be regarded as a generalization of the theorem 1.5 in \cite{Xu15} if $K=1$.\par

\begin{theorem}\label{Thm-1}
  Let $\omega$ be an abelian differential of the third kind having poles on a  Riemann surface  $M$ without a boundary (not necessarily compact), whose residues are all nonzero real numbers and whose real part is exact outside the set of poles of $\omega$. Denote $(\omega)=\sum_{i=1}^{I}(\alpha_{i}-1)Q_{i}-\sum_{n=1}^{N}P_{n}$. Then there exists a unique continuous function $\Phi$ on $M$ satisfied
 \begin{equation}\label{E-1}
\begin{cases}
\frac{4d\Phi}{\Phi(4-\Phi)}=\omega+\overline{\omega},\\
\Phi(p_{0})=\Phi_{0}\in (0,4),p_{0}\in M\setminus\{p_{1},\ldots,p_{N}\},
\end{cases}
\end{equation}
on $M\setminus\{p_{1},\ldots,p_{N}\}$.\par
Furthermore, suppose $K\in\{-1,0,1\}$, then
 $$g=\frac{4\Phi(4-\Phi)}{[4+(K-1)\Phi]^{2}}\omega\overline{\omega}$$
 is a conformal CSC-K metric on $M$ with  singularities (not necessary conical singularity). More precisely, \\
 (1) if $K=1$, at the zeros of $\omega$ the singular angles of $g$ are of the form $2\pi(ord_{p}(\omega)+1)$, at the poles of $\omega$ the singular angles of $g$ are of the form $2\pi Res_{p}(\omega)$ or $-2\pi Res_{p}(\omega)$ depending on the sign of $Res_{p}(\omega)$, and $Res_{p}(\omega)=1$ or $Res_{p}(\omega)=-1$ means that $p$ is a smooth point of $g$, i.e., $g$ represents divisor
 $$D=\sum_{i=1}^{I}(\alpha_{i}-1)Q_{i}+\sum_{n=1}^{N} (|Res_{p_{n}}\omega|-1)P_{n}$$\\
 (2) if $K=0$, at the zeros of $\omega$ the singular angles of $g$ are of the form $2\pi(ord_{p}(\omega)+1)$, and at the poles of $\omega$ the singular angles of $g$ are of the form $2\pi Res_{p}(\omega)$ if $Res_{p}(\omega)>0$. If $Res_{p}(\omega)<0$, $p$ is singularity but not a conical singularity.  $Res_{p}(\omega)=1$ means that $p$ is a smooth point of $g$;\\
 (3) if $K=-1$, at the zeros of $\omega$ the singular angles of $g$ are of the form $2\pi(ord_{p}(\omega)+1)$ if $\Phi(p)\neq2$, and at the poles of $\omega$ the singular angles of $g$ are of the form $2\pi Res_{p}(\omega)$ or $-2\pi Res_{p}(\omega)$ depending on the sign of $Res_{p}(\omega)$.  $Res_{p}(\omega)=1$ means that $p$ is a smooth point of $g$.
\end{theorem}

\begin{rem}
When we substitute $\omega$ by $-\omega$ in (\ref{E-1}), we get the same metric.
\end{rem}

\begin{rem}
We should point out that if $K\in\{-1,0\}$ and $M$ compact,  by the Gauss-Bonnet formula, the singularities of $g$ are not all conical singularities.
\end{rem}

In \cite{UY00}, the authors called $g$ reducible if the monodromy group of $g$ has at least one fixed point on $\overline{\mathbb{C}}$. By the results in \cite{Xu15}, a CSC-1 metric $g$ in the theorem \ref{Thm-1} is reducible and $\omega$  a character 1-form of $g$. In \cite{Xu15}, the authors proved that if $g$ is a  CSC-1 metric on $S^{2}$ with only two conical singularities then $g$ is a reducible metric  with the same conical angles.  By this fact,  as an application of theorem \ref{Thm-1}, we classify CSC-1 metrics on $S^{2}_{\{\alpha,\alpha\}}$.

\begin{theorem}\label{Thm-2}
Regard $S^{2}$ as $\mathbb{C}\cup\{\infty\}$. If $g$ is a conformal CSC-1 metric on $\mathbb{C}\cup\{\infty\}$ with conical singularities at $z=0$ and $z=\infty$ and conical angles $2\pi\alpha,2\pi\alpha(0<\alpha\neq1)$ respectively, then, up to a change of coordinate ($z\rightarrow pz,p\in \mathbb{C}\setminus\{0\}$ a constant), we have\\
(1) if $\alpha\overline{\in}\mathbb{Z}^{+}$, $g=\frac{4\alpha^{2}|z|^{2(\alpha-1)}}{(1+|z|^{2})^{2}}|dz|^{2};$ \\
(2) if $\alpha\in \mathbb{Z}^{+}$, $g=\frac{4\alpha^{2}|z|^{2(\alpha-1)}}{(1+|z^{\alpha}+b|^{2})^{2}}|dz|^{2},$ where $b\in\mathbb{R}$ is a constant.
\end{theorem}
\begin{rem}
Setting $w=\frac{1}{z}$, one can obtain the theorem II in \cite{Tr89}.

\end{rem}
\section{Proof of theorem \ref{Thm-1}}
In this section, we will give the proof of (1) in theorem \ref{Thm-1}. We left the proof of (2)(3) to readers.\par
Since $\omega+\overline{\omega}$ is exact on $ M\setminus\{p_{1},\ldots,p_{N}\}$, we suppose that
$$\omega+\overline{\omega}=df.$$
Since
$$\frac{4d\Phi}{\Phi(4-\Phi)}=d\ln\frac{\Phi}{4-\Phi},$$
then
$$ \ln\frac{\Phi}{4-\Phi}=f+A_{0},$$
i.e.,
$$\Phi=\frac{4e^{f+A_{0}}}{1+e^{f+A_{0}}},$$
where $A_{0}=\ln\frac{\Phi(p_{0})}{4-\Phi(p_{0})}-f(p_{0})=\ln\frac{\Phi_{0}}{4-\Phi_{0}}-f(p_{0})$.\par

(1) $\Phi$ can be continuously extended to $p_{n},1\leq n\leq N$. \par
We only need to prove $\Phi$ has limits at each point $p_{n},1\leq n\leq N$. Suppose $(D,z)$ is a local complex coordinate disk around some $p_{n},1\leq n\leq N$ with $z(p_{n})=0$ such that there are no other poles and zeros of $\omega$ in $D$. Suppose
\begin{equation}\label{w-1}
\omega=\frac{\lambda_{n}}{z}dz+df_{1},
\end{equation}
where $f_{1}$ is a holomorphic function on $D$. Then $(\omega+\overline{\omega})|_{D\setminus\{0\}}=d(\lambda_{n} \ln |z|^{2})+d(f_{1}+\overline{f_{1}})=df$, or
$f=\lambda_{n}\ln|z|^{2}+(f_{1}+\overline{f_{1}})+a^{*}$ on $D\setminus\{0\}$, where $a^{*}$ is a real constant. This implies
\begin{equation}\label{Exp-1}
\Phi=\frac{4|z|^{2\lambda_{n}}e^{(f_{1}+\overline{f_{1}})+a^{*}+A_{0}}}{1+|z|^{2\lambda_{n}}e^{(f_{1}+\overline{f_{1}})+a^{*}+A_{0}}}~ on ~D\setminus\{0\}.
\end{equation}
So if $\lambda_{n}>0$, $\lim_{z\rightarrow 0}\Phi=0$; if $\lambda_{n}<0$, $\lim_{z\rightarrow 0}\Phi=4$.\par

(2) $g=\frac{\Phi(4-\Phi)}{4}\omega\overline{\omega}$ is a Riemannian metric of constant curvature 1 on  $M\setminus\{q_{1},\ldots,q_{I},p_{1},\ldots,p_{N}\}$.\par

Since $0< \Phi< 4$ and $\omega\neq0$ on $M\setminus\{q_{1},\ldots,q_{I},p_{1},\ldots,p_{N}\}$, $g$ is a Rieamnnian metric.\par

Take a local complex coordinate domain $(X,z)$ on $M\setminus\{q_{1},\ldots,q_{I},p_{1},\ldots,p_{N}\}$ and suppose $\omega=\eta dz$ on $X$, where $\eta$ is a nonvanishing holomorphic function on $X$. Then $g|_{X}=\frac{\Phi(4-\Phi)|\eta(z)|^{2}}{4}|dz|^{2}$. A direct calculation shows the Gauss curvature of $g$ is $K\equiv 1$.\par

(3) $q_{i}, i=1,\ldots,I$ are conical singularities of $g$ with conical angles $2\pi(ord_{q_{i}}(\omega)+1),i=1,\ldots,I$ respectively, and $p_{n},n=1,\ldots,N$ are conical singularities of $g$ with conical angles $\pm2\pi Res_{p_{n}}(\omega),n=1,\ldots, N$ respectively.\par

Choose a local complex coordinate chart $(W,z)$ around $q_{i}$ with $z(q_{i})=0$ and suppose $\omega|_{W}=z^{\alpha_{i}-1}g(z)dz$, where $g$ is a holomorphic function on $W$ with $g(0)\neq0$. Then
 $$g|_{W\setminus\{0\}}=\frac{\Phi(4-\Phi)}{4}|z|^{2(\alpha_{i}-1)}|g(z)|^{2}|dz|^{2}.$$
So $q_{i}, i=1,\ldots,I$ are conical singularities of $g$ with conical angles $2\pi(ord_{q_{i}}(\omega)+1),i=1,\ldots,I$ respectively.\par

By (\ref{w-1}), (\ref{Exp-1}), one can easily prove $p_{n},n=1,\ldots,N$ are conical singularities of $g$ with conical angles $\pm2\pi Res_{p_{n}}(\omega),n=1,\ldots, N$ respectively. Moreover, $Res_{p}(\omega)=1$ or $Res_{p}(\omega)=-1$ means that $p$ is a smooth point of $g$.

\section{Classification of conformal CSC-1 metrics on $S^{2}_{\{\alpha,\alpha\}}$}

In this section, as an application of the theorem \ref{Thm-1}, we will classify CSC-1 metrics on $S^{2}_{\{\alpha,\alpha\}}$. First, we  proof the standard expressions of a kind of meromorphic 1-forms on $S^{2}$.  Regard $S^{2}$ as $\mathbb{C}\cup\{\infty\}$.
\begin{lem}\label{L-1}
Suppose $\omega$ be an abelian differential of the third kind having poles on $\mathbb{C}\cup\{\infty\}$, whose residues are all nonzero real numbers. \\
(1) If $(\omega)=-0-\infty$, then $\omega=\frac{\lambda}{z}dz$, where $\lambda=Res_{0}\omega$;\\
(2) If $(\omega)=(\alpha-1)\cdot0-\infty-\sum_{i=1}^{\alpha}P_{i},\alpha\geq2$, and $Res_{p_{i}}\omega=1,i=1,\ldots,\alpha$, then $\omega=\frac{\alpha z^{\alpha-1}}{z^{\alpha}+1}dz$, up to a change of coordinate ($z\rightarrow pz,p\in \mathbb{C}\setminus\{0\}$ a constant). \\
(3) If $(\omega)=(\alpha-1)\cdot0+(\alpha-1)\cdot\infty-\sum_{i=1}^{\alpha}P_{i}-\sum_{j=1}^{\alpha}Q_{j},\alpha\geq2$,  and $Res_{p_{i}}\omega=1,Res_{q_{i}}\omega=-1, i=1,\ldots,\alpha$, then  $\omega=\frac{\alpha (a- 1) z^{\alpha-1}}{(z^{\alpha}+a)(z^{\alpha}+1)}dz$, up to a change of coordinate ($z\rightarrow pz,p\in \mathbb{C}\setminus\{0\}$ a constant), where $a$ is a complex constant and $a\neq0,1$.
\end{lem}
\begin{proof}
(1) It is an obvious fact.\\
(2) Suppose $p_{i}=a_{i}\in\mathbb{C}\setminus\{0\},i=1,\ldots,\alpha,\mu\neq 0$ and
$$\omega=(\sum_{i=1}^{\alpha}\frac{1}{z-a_{i}})dz=\frac{\alpha \mu z^{\alpha-1}}{\prod_{i=1}^{\alpha}(z-a_{i})}dz.$$
Set $t(z)=\prod_{i=1}^{\alpha}(z-a_{i})$, then
$$t'(z)=\alpha \mu z^{\alpha-1}.$$
So $\mu=1,t(z)= z^{\alpha}+a, \omega=\frac{\alpha  z^{\alpha-1}}{z^{\alpha}+a}dz, a\in\mathbb{C}\setminus\{0\}$ a constant.\\
Set $z=a^{\frac{1}{\alpha}}w$, then
$$\omega=\frac{\alpha  w^{\alpha-1}}{w^{\alpha}+1}dw.$$
(3) Suppose $p_{i}=a_{i},q_{i}=b_{i}\in\mathbb{C}\setminus\{0\},i=1,\ldots,\alpha,\mu\neq0$ and
$$\omega=(\sum_{i=1}^{\alpha}\frac{1}{z-a_{i}}-\sum_{i=1}^{\alpha}\frac{1}{z-b_{i}} )dz=\frac{\alpha \mu z^{\alpha-1}}{\prod_{i=1}^{\alpha}(z-a_{i})\prod_{i=1}^{\alpha}(z-b_{i})}dz.$$

Set $t(z)=\prod_{i=1}^{\alpha}(z-a_{i}),s(z)=\prod_{i=1}^{\alpha}(z-b_{i})$, then
\begin{equation}\label{L-E-1}
t'(z)s(z)-t(z)s'(z)=\alpha \mu z^{\alpha-1}.
\end{equation}

Set
$$t(z)=\sum_{i=0}^{\alpha}\omega_{i}z^{i},s(z)=\sum_{i=0}^{\alpha}\sigma_{i}z^{i},$$
where $\omega_{\alpha}=\sigma_{\alpha}=1$ and $\omega_{0},\sigma_{0}\neq 0$.
Then (\ref{L-E-1}) is equivalent to
\begin{equation}\label{L-E-2}
\sum_{i=1}^{\alpha}\sum_{j=0}^{\alpha}i(\omega_{i}\sigma_{j}-\sigma_{i}\omega_{j})z^{i+j-1}=\alpha \mu z^{\alpha-1}.
\end{equation}
Considering the coefficient of  $z^{2\alpha-1}$ in (\ref{L-E-2}),
$$\sigma_{\alpha-1}-\omega_{\alpha-1}=0.$$
Considering the coefficient of  $z^{2\alpha-2}$ in (\ref{L-E-2})
$$\sigma_{\alpha-2}-\omega_{\alpha-2}=0.$$
......\\
Considering the coefficient of  $z^{\alpha+1}$ in (\ref{L-E-2})
$$\sigma_{1}-\omega_{1}=0.$$
So
$$t(z)=F(z)+\omega_{0},s(z)=F(z)+\sigma_{0},$$
where $F(z)=z^{\alpha}+\omega_{\alpha-1}z^{\alpha-1}+\ldots+\omega_{1}z$.\\
By (\ref{L-E-1}),
$$F'(z)(\sigma_{0}-\omega_{0})=\alpha \mu z^{\alpha-1}.$$
So $\omega_{1}=\ldots=\omega_{\alpha-1}=0,\sigma_{0}-\omega_{0}=\mu$, i.e.,
$$t(z)=z^{\alpha}+\omega_{0}, s(z)=z^{\alpha}+\sigma_{0}.$$
Thus
$$\omega=\frac{\alpha (\sigma_{0}-\omega_{0}) z^{\alpha-1}}{(z^{\alpha}+\omega_{0})(z^{\alpha}+\sigma_{0})}dz.$$
Set $z=\omega_{0}^{\frac{1}{\alpha}}w$, then
$$\omega=\frac{\alpha (a-1)w^{\alpha-1}}{(w^{\alpha}+1)(\omega^{\alpha}+a)}dw,$$
where $a=\frac{\sigma_{0}}{\omega_{0}}\neq0,1$ is a complex constant.
\end{proof}
\textbf{Poof of theorem \ref{Thm-2}}.\par
(1) Set $\omega_{}=\frac{\alpha}{z}dz$. By  direct calculation,
$$g=\frac{4\alpha^{2}e^{A_{0}}|z|^{2(\alpha-1)}}{(1+e^{A_{0}}|z|^{2\alpha})^{2}}|dz|^{2},$$
where $A_{0}\in\mathbb{R}$ is a constant.\par
Set $z=e^{-\frac{A_{0}}{2\alpha}}w$, then
$$g=\frac{4\alpha^{2}|w|^{2(\alpha-1)}}{(1+|w|^{2\alpha})^{2}}|dw|^{2}.$$

(2) Set $\omega=\frac{\alpha z^{\alpha-1}}{z^{\alpha}+1}dz$ and $2\leq\alpha\in\mathbb{Z}^{+}$.
By  direct calculation,
$$g=\frac{4\alpha^{2}e^{A_{0}}|z|^{2(\alpha-1)}}{(1+e^{A_{0}}|z^{\alpha}+1|^{2})^{2}}|dz|^{2},$$
where $A_{0}\in\mathbb{R}$ is a constant.\par
Set $z=e^{\frac{-A_{0}}{2\alpha}}w$, then
$$g=\frac{4\alpha^{2}|w|^{2(\alpha-1)}}{(1+|w^{\alpha}+b|^{2})^{2}}|dw|^{2},$$
where $b=e^{\frac{A_{0}}{2}}\in\mathbb{R}\setminus\{0\}$.\\

(3) Set $\omega=\frac{\alpha (a-1) z^{\alpha-1}}{(z^{\alpha}+a)(z^{\alpha}+1)}dz$, where $a$ is a complex constant and $a\neq0,1$.
By  direct calculation,
$$g=\frac{4\alpha^{2}e^{A_{0}}|a-1|^{2}|z|^{2(\alpha-1)}}{(|z^{\alpha}+a|^{2}+e^{A_{0}}|z^{\alpha}+1|^{2})^{2}}|dz|^{2},$$
where $A_{0}\in\mathbb{R}$ is a constant.\par
Denote $e^{A_{0}}$ by $\lambda^{2}$, then
\begin{equation*}
\begin{aligned}
g&=\frac{4\alpha^{2}\lambda^{2}|a-1|^{2}|z|^{2(\alpha-1)}}{(|z^{\alpha}+a|^{2}+\lambda^{2}|z^{\alpha}+1|^{2})^{2}}|dz|^{2}\\
&=\frac{4\alpha^{2}\lambda^{2}|a-1|^{2}|z|^{2(\alpha-1)}}{(1+\lambda^{2})^{2}[|z^{\alpha}+\frac{a+\lambda^{2}}{1+\lambda^{2}}|^{2}+\frac{\lambda^{2}|a-1|^{2}}{(1+\lambda^{2})^{2}}]^{2}}|dz|^{2}.
\end{aligned}
\end{equation*}
Set $z=pw$, where $|p^{\alpha}|=\frac{\lambda|a-1|}{1+\lambda^{2}}$ such that $b=\frac{a+\lambda^{2}}{p^{\alpha}(1+\lambda^{2})}\in\mathbb{R}$, then
$$g=\frac{4\alpha^{2}|w|^{2(\alpha-1)}}{(1+|w^{\alpha}+b|^{2})^{2}}|dw|^{2}.$$

\textbf{Acknowledgments.} This work was accomplished when the author visited IGP (Institute of Geometry and Physics) in 2022. The author would like to express his sincere gratitude to Professor Bin Xu for many helpful and valuable suggestions. Thank also to Professors Xiuxiong Chen, Qing Chen, Bing Wang and Xiaowei Xu for their warm hospital. Also to my advisor Professor Yingyi Wu for his warm encouragements during the past  years.

\end{document}